\documentclass[11pt]{article}
\usepackage{amsmath}
\usepackage{amssymb}
\usepackage[amsmath,thmmarks,hyperref]{ntheorem}
\usepackage{mathrsfs}
\usepackage{eucal}
\usepackage[all]{xy}
\usepackage{float}
\usepackage{stmaryrd}
\usepackage[bookmarksnumbered]{hyperref}
\usepackage{rewrite}
\usepackage{todonotes}
\usepackage{parskip}
\usepackage{enumitem}
\usepackage[textwidth=14.5cm, textheight=22cm]{geometry}
\usepackage{setspace}

\let\OLDthebibliography\thebibliography
\renewcommand\thebibliography[1]{
 \OLDthebibliography{#1}
 \setlength{\parskip}{1.4pt}
 \setlength{\itemsep}{0pt plus 0.2ex}
}

\DeclareMathOperator{\tinc}{inc}
\DeclareMathOperator{\tfix}{fix}

\DeclareMathOperator{\tails}{tails}
\DeclareMathOperator{\pullback}{pb}
\DeclareMathOperator{\classify}{B}
\DeclareMathOperator{\cell}{--cell--}

\DeclareMathOperator{\Top}{top}
\DeclareMathOperator{\Spectra}{Sp}
\newcommand{\Sp}{{ \Spectra^\ocal }}
\newcommand{\osp}{{ O(2) \Spectra^\ocal_\qq }}
\newcommand{\tsp}{{ \torus \Spectra^\ocal_\qq }}
\newcommand{\qsp}{{ \Spectra^\ocal_\qq }}
\newcommand{\wsp}{{ \Spectra^\ocal_\qq[W] }}
\newcommand{\sphspec}{{ \mathbb{S} }}

\newcommand{\torus}{{ \mathbb{T} }}
\newcommand{\dihedsp}{{ L_{e_\dcal \sphspec} \osp }}

\title{Rational $O(2)$--Equivariant Spectra}
\author{David Barnes \\
\footnotesize{Queen's University Belfast\footnote{
School of Mathematics and Physics,
Pure Mathematics Research Centre,
University Road,
Belfast, BT7 1NN,
United Kingdom
}} \\
\footnotesize{d.barnes@qub.ac.uk}
}
\date{July 31, 2016}

\begin{document}
\maketitle

\begin{abstract}
\noindent
The category of rational $O(2)$--equivariant cohomology theories
has an algebraic model $\acal(O(2))$, as established
by work of Greenlees. That is, there is an equivalence of categories
between the homotopy category of rational $O(2)$--equivariant spectra
and the derived category of the abelian model $D\acal(O(2))$.
In this paper we lift this equivalence of homotopy categories to the level of
Quillen equivalences of model categories.
This Quillen equivalence is also compatible with the
Adams short exact sequence of the algebraic model.
\end{abstract}

{\small \noindent \textbf{MSC:} 55N91, 55P42, 55P60 \\
\noindent \textbf{Keywords:} equivariant spectra, model categories, right Bousfield localisation,
ring spectra, algebraic models}

\pdfbookmark[1]{Contents}{toc}
\tableofcontents

\section{Introduction}

Equivariant cohomology theories are a fundamental tool for studying spaces with
a $G$--action.  To study these cohomology theories,
it is helpful to understand the $G$--spectra that represent them.
The homotopy category of $G$--spectra is particularly complicated.
It contains all the information of
the stable homotopy category, as well as equivariant information such as the
Burnside ring of $G$ and the group
cohomology of $G$. A standard and fruitful method to make this category easier
to study is to work rationally. Since the rational stable homotopy category
is equivalent to the category of graded rational vector spaces, we have removed
most of the topological complexity.
However, much of the interesting behaviour that comes from the group is preserved and made tractable.

The category of rational $G$--spectra has been classified
(via Quillen equivalences) in terms of
a simple algebraic category for a number of groups.
The case of finite groups has been covered by the author
\cite{barnesfinite} and Kedziorek \cite{kedziorekexceptional}.
The circle group $\torus=SO(2)$ has recently been completed
by the author, Greenlees, Kedziorek and Shipley in \cite{BGKS}.
The case of a torus group is considered in
Greenlees and Shipley \cite{tnqcore}.

In this paper we focus on the group $O(2)$, which is the simplest non-commutative
non-finite compact Lie group.
The paper \cite{greo2} gives an abelian model $\acal(O(2))$
for the homotopy category of rational $O(2)$--spectra.
By considering objects with a differential in $\acal(O(2))$
we can construct a model category $d \acal(O(2))$
and a Quillen equivalence between $d \acal(O(2))$ and
the model category of rational $O(2)$--spectra.
We call $d \acal(O(2))$ the \textbf{algebraic model for rational $O(2)$--spectra}.
This Quillen equivalence gives us a triangulated equivalence of the homotopy categories
and also tells us that all further homotopical
structures (such as homotopy limits or Toda brackets) are preserved by this equivalence.

The algebraic model is explicit and manageable so that constructing
objects or maps is straightforward.
Furthermore, there is an Adams
exact sequence relating maps in the homotopy category of rational $O(2)$--spectra
to the algebraic model, see
Theorems \ref{thm:toraladams} and \ref{thm:dihedraladams}.
The algebraic model splits into the product of two simpler categories.
The first, called the \textbf{toral part} $d \acal(\ccal)$,
comes from (the homotopy category of) the algebraic model for
$\torus$ along with a skewed action of $W=O(2)/\torus$,
see Definition \ref{def:toralmodel}.
Our work will also clarify an imprecision in \cite{greo2} regarding
the behaviour of $W$ on the toral part, see Remark \ref{rmk:behaviour}.
The second, called the \textbf{dihedral part} $dg \acal(\dcal)$,
behaves much more like the
case of a finite group (or an exceptional subgroup as in \cite{kedziorekexceptional}), see Definition \ref{def:dihedralmodel}.
The main theorem can be phrased as follows, where we denote rational $O(2)$--equivariant
(orthogonal) spectra by $\osp$.

\begin{theorem}
There is a zig-zag of Quillen equivalences between
\[
\osp
\quad \textrm{and} \quad
d \acal(\ccal) \times dg \acal(\dcal).
\]
\end{theorem}

This paper represents the prototype for other extensions
of a torus by a finite group.
Let $G$ be an extension of $\torus^r$ by a finite group,
then there is a notion of `toral' $G$--spectra
extending our notion of toral $O(2)$--spectra.
A study of the homotopy category of such spectra appears in
\cite{gretoralsupport} and that paper gives an abelian model for the homotopy category.
A similar method to Section \ref{sec:toral} and an extension of the results of
\cite{barnesmonoidal} should provide a classification
of toral $G$--spectra in terms of the algebraic model built from the
abelian model. The key fact is that a map of toral $G$--spectra is a weak equivalence
if and only if it forgets to a weak equivalence of $\torus^r$--spectra.


\subsection{Organisation}

The model category of rational
$O(2)$--spectra is recalled in Section \ref{sec:spectra}.
Theorem \ref{thm:mainsplitting}
splits the model category into two parts:
the model category of `toral spectra' and
the model category of `dihedral spectra', see Definition \ref{def:thecategories}.
This is the model category version of the
fact that the homotopy category splits into two pieces.

The model category of toral spectra only has homotopical information
coming from the finite cyclic groups and $\torus$ and hence behaves very much like
the model category of rational $\torus$--spectra.
In Section \ref{sec:toralmodel} we introduce the algebraic model for toral spectra.
The classification of toral spectra in terms of the algebraic model is given in
Section \ref{sec:toral}.
The method is an extension of the method of \cite{BGKS}.
The extra difficulty is accounting for the action of the
Weyl group $W= W_{O(2)} \torus =O(2)/\torus$.
The essential point is that the $\torus$--fixed points of a toral $O(2)$--spectrum
have the structure of a spectrum with $W$--action.

The model category of dihedral spectra contains the
homotopical information of rational $O(2)$--spectra
generated by the dihedral subgroups and $O(2)$.
In Section \ref{sec:dihedral} we introduce the algebraic model for dihedral spectra
and give the classification.
This part of the work is simpler than the toral case as
all the homotopical information is concentrated in degree zero,
see Lemma \ref{lem:dihedralcalc}.

\section{Rational \texorpdfstring{$O(2)$}{O(2)}--Spectra}\label{sec:spectra}

In this section we introduce a model category for rational $O(2)$--spectra and
show that it splits into the product of two localisations.
We also give some information on the group $O(2)$ and introduce some basic results on
$O(2)$--spectra.

\subsection{The group \texorpdfstring{$O(2)$}{O(2)} and model structures}

Let $D^h_{2n}$ denote the dihedral subgroup of order
$2n$ containing $h$, where $h$ is an element of $O(2) \setminus SO(2)$.
The closed subgroups of $O(2)$ are $O(2)$, $SO(2)$, the finite cyclic groups $C_n$ ($n \geqslant 1$)
and the finite dihedral
groups $D^h_{2n}$ for varying $h$.

For $H$ a closed subgroup of $O(2)$ we let $N_{O(2)}(H)$ denote the normaliser of $H$ in $O(2)$:
So $N_{O(2)}(H)$ is the largest subgroup of $O(2)$ in which $H$ is normal.
The Weyl group of $H$ in $O(2)$ is
\[
W_{O(2)}(H):=N_{O(2)}(H)/H
\]
The Weyl group
of $O(2)$ is the trivial group.
The Weyl group of $\torus=SO(2)$ is the group of order two,
which we call $W$
\[
W:=O(2)/SO(2)=W_{O(2)}(SO(2)).
\]
The normaliser of $D^h_{2n}$ in $O(2)$ is $D^h_{4n}$,
thus the Weyl group of $D_{2n}^h$ is isomorphic to $W$.
The finite cyclic groups are normal, hence the
Weyl group of $C_n$ is $O(2)/C_n \cong O(2)$.

Following \cite[Chapter V, Section 2]{lms86},
define $\fcal O(2)$ to be the set
of those subgroups of $O(2)$
with finite index in their normaliser
equipped with the Hausdorff topology.
This is an $O(2)$--space via the conjugation
action of $O(2)$ on its subgroups.
Let $C(\fcal O(2)/O(2), \qq)$ be the ring of continuous maps
from $\fcal O(2)/O(2)$ to $\qq$ considered as
a discrete space. By work of tom Dieck, see \cite[Lemma 2.10]{lms86},
there is an isomorphism of rings
\[
A(O(2)) \otimes \qq :=[\sphspec,\sphspec]^{O(2)} \otimes \qq \overset{\cong}{\longrightarrow}
C(\fcal O(2)/O(2), \qq)
\]
which sends $f \otimes q$ to
$(H) \mapsto q\deg(\Phi^H f)$ (the degree of the
$H$--fixed points of the map $f \co \sphspec \to \sphspec$).
We draw $\fcal O(2)/O(2)$
below as Figure \ref{phig}.
We will sometimes write $D_{2n}$ for $(D^h_{2n})$ the conjugacy class of
$D^h_{2n}$. The point $O(2)$ is a limit point
of this space.

\begin{figure}[!hbt]
\begin{center}
\setlength{\unitlength}{1cm}
\framebox[0.8\textwidth]{
\begin{picture}(6.5,3)(2,1)

\put(0,2){$\bullet$}

\put(0,3){$\bullet$}

\put(10,2){$\bullet$}
\put(5,2){$\bullet$}
\put(2.5,2){$\bullet$}
\put(3.3,2){$\bullet$}
\put(1.1,2.0){$\cdots$}

\put(-0.2,1.5){$O(2)$}
\put(-0.2,3.5){$SO(2)$}

\put(9.8,1.5){$D_{2}$}
\put(4.8,1.5){$D_{4}$}
\put(3.1,1.5){$D_{6}$}
\put(2.3,1.5){$D_{8}$}

\end{picture}}
\end{center}
\caption{\label{phig} $\fcal O(2)/O(2)$.}
\end{figure}

\begin{definition}
Define $\ccal$ to be the set consisting of the finite cyclic groups and $\torus$.
This set is a family in the sense that it is closed under conjugation and taking subgroups.
Let $\dcal$ be the complement of $\ccal$ in the set of all (closed) subgroups
of $O(2)$.
\end{definition}

\begin{definition}\label{def:o2idem}
We define idempotents of $C(\fcal O(2)/O(2), \qq)$ as follows:
$e_\ccal$ is the characteristic function of $\torus$, $e_\dcal :=e_\ccal-1$
and $e_n$ is the characteristic function of $D_{2n}$ for each $n \geqslant 1$.
We also let $f_n = e_\dcal-\Sigma_{k=1}^{n-1} e_k$.
\end{definition}

Our base category of spectra $\Sp$ is the category of orthogonal spectra,
equipped with the stable model structure.
Let $O(2) \Sp$ be the model category of $O(2)$--equivariant orthogonal spectra
(defined over a complete $O(2)$--universe $\ucal$).
This is a proper, cellular, stable model structure
where weak equivalences are those maps $f$ such that $\pi_*^H(f)$ is an isomorphism
for all closed subgroups $H$ of $O(2)$.
See \cite{mm02} for details.

Similarly we have $\torus \Sp$, the model category of $\torus$--equivariant orthogonal spectra.
We will index this category of spectra over the universe $i^* \ucal$
(so our $\torus$--equivariant spectra are indexed on $\torus$--representations
of the form $i^* V$, for $V$ an $O(2)$--representation). This non-standard choice
of universe is justified by \cite[Section V.2 and Remark 1.10]{mm02}. In particular
the homotopy category is the usual $\torus$--equivariant stable homotopy category.
One advantage of this convention is that the
forgetful functor from $O(2)$--spectra to $\torus$--spectra is given
by $(i^* X)(i^*V): = i^*(X(V))$, that is, the space $X(V)$ with $O(2)$--action
forgotten to an $\torus$--action, for $X$ an $O(2)$--spectrum.

Following \cite[Section 5]{barnessplitting} and using \cite[Theorem IV.6.3]{mm02},
we can create a new model structure on
$O(2) \Sp$ by localising at a rational sphere spectrum $S^0 \qq$.
The spectrum $S^0 \qq$ can be built as a non-equivariant spectrum and inflated
to an $O(2)$--spectrum or built directly in $O(2)$--spectra.
We call the weak equivalences of the localised model structure
\textbf{rational equivalences}: those maps $f$
such that $\pi_*^H(f) \otimes \qq$ is an isomorphism for all closed subgroups $H$
of $O(2)$.
We call this the \textbf{rational model structures}.
Analogous model structures exist for $\torus$--spectra and non-equivariant spectra.

\begin{definition}
Let $\osp$ be the category of $O(2)$--equivariant orthogonal spectra
equipped with the rational model structure. This category of spectra is indexed on
the complete $O(2)$--universe $\ucal$.

Let $\tsp$ be the category of $\torus$--equivariant orthogonal spectra
equipped with the rational model structure. This category of spectra is indexed on
the universe $i^* \ucal$.

Let $\qsp$ be the category of orthogonal spectra
equipped with the rational model structure. This category of spectra is indexed on
the universe $\rr^\infty$.
\end{definition}

We will also have cause to use a category of naive $W=O(2)/\torus$--equivariant spectra.
\begin{definition}
Let $\wsp$ denote the category of $W$--objects and $W$--maps in $\Sp$, indexed
on the universe $\rr^\infty$.
This category is equipped with the `free' model structure lifted from $\qsp$,
using the functor $W_+ \smashprod (-)$.
Hence a map is a weak equivalence or fibration if it is so in $\qsp$
when the $W$--action is forgotten.
\end{definition}

\subsection{Splitting rational \texorpdfstring{$O(2)$}{O(2)}--spectra}\label{sec:split}

We know by \cite{greo2} and \cite[Section 6]{barnessplitting} that the
homotopy theory of rational $O(2)$--spectra splits into two pieces.
Using the idempotents of Definition \ref{def:o2idem}
we define $e_\ccal \sphspec$ as the homotopy colimit (mapping telescope)
of
\[
\xymatrix{
\sphspec \ar[r]^{e_\ccal} & \sphspec \ar[r]^{e_\ccal} & \sphspec \ar[r]^{e_\ccal} & \dots
}
\]
and we require that this spectrum be cofibrant (either by choice of construction
or by replacing it with a cofibrant replacement).
Similarly we have $(1-e_\ccal) \sphspec \simeq e_\dcal \sphspec$.

We can then Bousfield localise
the model category of rational $O(2)$--spectra at these objects
to obtain $L_{e_\ccal \sphspec} \osp$
and $L_{e_\dcal \sphspec} \osp$ using \cite[Section IV.6]{mm02}.
The weak equivalences of
$L_{e_\ccal \sphspec} \osp$ are those maps $f$ such that
$e_\ccal \sphspec \smashprod f$ is a rational equivalence and similarly so for
$L_{e_\dcal \sphspec} \osp$.
These are cofibrantly generated, proper, simplicial stable model categories.
The result below is \cite[Corollary 6.3]{barnessplitting}.

\begin{theorem}\label{thm:mainsplitting}
The adjoint pair of the diagonal functor $\Delta$
and the product functor $\Pi$
induces a symmetric monoidal Quillen equivalence.
\[
\Delta :
\osp
\adjunct
L_{e_\ccal \sphspec}  \osp \times L_{e_\dcal \sphspec}  \osp
: \Pi
\]
\end{theorem}

We can identify these localised homotopy categories more clearly.
Note that for any $X$ and $Y$ in $ \osp$, the abelian group
$[X,Y]^{O(2)}_\qq $ is a module
over the ring $[\sphspec,\sphspec]^{O(2)}_\qq$ via the smash product.
Hence we have the following isomorphisms of sets of maps in the
homotopy category.
\[
\renewcommand*{\arraystretch}{1.3}
\begin{array}{rcl}
[X,Y]^{O(2)}_\qq
& \cong &
e_\ccal [X,Y]^{O(2)}_\qq \times e_\dcal [X,Y]^{O(2)}_\qq \\
e_\ccal [X,Y]^{O(2)}_\qq
& \cong &
\ho L_{e_\ccal \sphspec} \osp (X,Y)  \\

e_\dcal [X,Y]^{O(2)}_\qq
& \cong &
\ho L_{e_\dcal \sphspec} \osp (X,Y)
\end{array}
\]

We can improve our description of $L_{e_\ccal \sphspec}  \osp$
by describing $e_\ccal \sphspec$  in terms of a suspension spectrum.
Let $E \ccal$ denote the universal $O(2)$--space
corresponding to the family $\ccal$,
so $E \ccal^H$ is non-equivariantly contractible
for each $H \in \ccal$ and is the empty set for $H \notin \ccal$.
This is an $O(2)$--CW--complex and is built from cells of the form
$O(2)/K_+$ for $K \in \ccal$.
Define $\widetilde{E} \ccal$ via the cofibre sequence of
$O(2)$--spaces,
\[
E \ccal_+ \to S^0 \to \widetilde{E} \ccal.
\]
By considering geometric fixed points, it is easy to check that
the composite map $E \ccal_+ \to S \to e_\ccal \sphspec$ of $O(2)$--spectra
is a weak equivalence. It follows that the weak equivalences of
$L_{e_\ccal \sphspec} \osp$ are those maps $f$ such that
$E \ccal_+ \smashprod f$ is a rational equivalence.
By \cite[Proposition IV.6.7]{mm02} it follows that the
weak equivalences are those maps $f$ such that
$i^* f$ is a weak equivalence of rational $\torus$--spectra.

We will find it convenient to use a slightly different model structure on
$\osp$ to model $\ho L_{e_\ccal \sphspec} \osp$.
We take the following construction from \cite[Theorem IV.6.5]{mm02}.
The (non-rational) stable model structure on $O(2)$--spectra
has sets of generating cofibrations and acyclic cofibrations
obtained by applying
the shifted suspension functors $F_V$ to spaces of the form
$O(2)/H_+ \smashprod A$ for $H \leqslant O(2)$, $V$ a representation of $O(2)$
and $A$ either a sphere or a disc.
If we restrict ourselves to only those with
$H \leqslant \torus$, we obtain a new model structure
on $O(2)$--spectra that is stable and cellular.
In particular, a map $f$ is a weak equivalence or fibration in
this new model structure if and only if $i^* f$
is a weak equivalence or fibration of $\torus$--spectra.
We rationalise this model structure as above and
denote it $\ccal \osp$.
From the descriptions of the weak equivalences
it follows immediately that the identity functor from
$\ccal \osp$ to
$L_{e_\ccal \sphspec} \osp$ is the left adjoint of a Quillen equivalence.

\begin{definition}\label{def:thecategories}
We call $\ccal \osp$
the model category of \textbf{toral $O(2)$--spectra}.
We call $L_{e_\dcal \sphspec} \osp$
the model category of \textbf{dihedral $O(2)$--spectra}.
\end{definition}

We rephrase the splitting result.
\begin{corollary}\label{cor:finalsplit}
The model category of rational $O(2)$--spectra is
Quillen equivalent to
\[
\ccal \osp \times L_{e_\dcal \sphspec}  \osp.
\]
\end{corollary}

In Section \ref{subsec:fixedpoints}
we will make much use of the $\torus$--fixed points functor, so we introduce
that functor and discuss how it acts on the model category $\ccal \osp$.
The functor $(-)^{\torus}$ of \cite[Section V.3]{mm02} first restricts
an $O(2)$--spectrum indexed on a complete $O(2)$--universe $\ucal$ to
$\rr^\infty$, then applies the space--level fixed point functor levelwise.

\begin{lemma}\label{lem:fixed}
The (categorical) $\torus$--fixed points functor induces a Quillen pair
\[
\varepsilon^* :
\wsp
\adjunct
\ccal \osp
: (-)^{\torus}
\]
\end{lemma}
\begin{proof}
We first consider the adjunction before rationalising.
On the left a map is a fibration if and only if it forgets to a fibration
of non-equivariant spectra. A map on the right is a fibration if and
only if it forgets to a fibration of $\torus$--spectra.
The functor $(-)^{\torus}$ is a right Quillen functor from
$\torus$--spectra to spectra by \cite[Proposition 3.4]{mm02}
hence we have a Quillen pair before localisation.

The adjunction extends to the rationalised categories as
the rational sphere spectrum $S^0 \qq$ in each category is given by applying the
appropriate inflation functor from non-equivariant spectra.
\end{proof}

The forgetful and fixed points functors interact well,
as the commutative diagram below shows.
Analogues of this diagram will appear throughout
Section \ref{sec:toral}.

\begin{proposition}\label{prop:commforgetful}
There is a diagram of Quillen functors as below,
in which both the square of fixed point and forgetful functors commute,
as does the square of inflation and forgetful functors.
\[\xymatrix@R+0.1cm@C+2cm{
\wsp
\ar@<0.1cm>[r]^{\varepsilon^* }
\ar@<0.1cm>[d]^{i^*}
&
\ccal \osp
\ar@<0.1cm>[d]^{i^*}
\ar@<0.1cm>[l]^{(-)^{\torus}}
\\
\qsp
\ar@<0.1cm>[r]^{\varepsilon^*}
&
\tsp
\ar@<0.1cm>[l]^{(-)^{\torus}}
}\]
Then both functors $i^*$ preserve weak equivalences, cofibrations and fibrations.
Furthermore if $i^* f$ is a weak equivalence (or fibration), then
$f$ is a a weak equivalence (or fibration).
\end{proposition}
\begin{proof}
The two commutativity statements follow directly from the definitions.
We have already discussed weak equivalences and fibrations for both functors $i^*$
in the non-rationalised case. These extend to the rational versions
as $S^0 \qq$ can be constructed in $\Sp$ and then inflated to
an equivariant spectrum in any of the other three categories.

For cofibrations, consider a generating cofibration of
$\ccal \osp$,
\[
F_V \left( O(2)/H_+ \smashprod S^{n-1}_+ \right) \longrightarrow
F_V \left( O(2)/H_+ \smashprod D^{n}_+ \right)
\]
for $H$ a subgroup of $\torus$ and $V$ a representation of $O(2)$.
Applying $i^*$ to this gives
\[
F_{i^* V} \left( (\torus/H_+ \vee j^* \torus/H_+) \smashprod S^{n-1}_+ \right)
\longrightarrow
F_{i^* V} \left( (\torus/H_+ \vee j^* \torus/H_+) \smashprod D^{n}_+ \right)
\]
where $j^* \torus/H_+$ denotes the space $\torus/H$ but with the inverse action of
$\torus$ (so $t \in \torus$ acts by $t^{-1}$).
Since this map is a cofibration of $\torus$--spectra, the statement follows.
A similar argument holds in the case of $\wsp$.
\end{proof}

\section{The toral model}\label{sec:toralmodel}

In this section we define $d\acal(\ccal)$, the algebraic model for toral spectra
and explain how it relates to $d\acal(\torus)$, the algebraic model for rational $\torus$--spectra.

\subsection{The model \texorpdfstring{$\acal(\torus)$}{A(T)}}\label{subsec:the model}
The algebraic category for the homotopy category of rational $\torus$--spectra
is established in \cite{gre99}.
We adapt that category to the toral case and explain how to relate it to the
$\torus$--case. Our starting point is the category of
chain complexes with an action of $W$.

\begin{definition}
The category $\ch(\qq[W])$ is the category
of rational chain complexes that have an action of the group of
order two. This is a monoidal category
with tensor product given by tensoring over $\qq$
and using the diagonal $W$--action. The unit
of this product is $\qq$ in degree zero with trivial
$W$--action.
\end{definition}

There is a proper cofibrantly generated model structure
on this category by \cite[Proposition 4.2.13]{hov99}.
The fibrations are the surjections and the weak equivalences
are the homology isomorphisms. The cofibrations are dimensionwise split injections with
cofibrant cokernel.
Let $S^{n-1}$ be the chain complex with $\qq$ in degree $n-1 \in \zz$ and zeroes
elsewhere and the $D^n$ be the chain complex with $\qq$ in degrees
$n$ and $n-1$ (with the identity as the differential between these degrees)
and zeroes elsewhere.
The generating cofibrations are given by the inclusion maps
$S^{n-1} \otimes \qq[W] \to D^n \otimes \qq[W]$
and the acyclic cofibrations are given by
$0 \to D^n \otimes \qq[W]$ for $n \in \zz$.

Since $\qq$ is a retract of $\qq[W]$ we see that
$S^{n-1} \otimes \qq \to D^n \otimes \qq$ is also a
cofibration of $\ch(\qq[W])$. Hence the cofibrant
objects do not have to be $W$--free.

The forgetful functor from $\ch(\qq[W])$ to
$\ch(\qq)$ is the right adjoint of a
strong monoidal Quillen pair. The left adjoint
sends a chain complex
$X$ to $X \oplus X$ with $W$ acting as the exchange of factors map.

It is routine to check that this category is a symmetric monoidal
model category that satisfies the monoid axiom.
We construct a commutative monoid in $\ch(\qq[W])$.

\begin{definition}
Let $\ocal_\fcal$ be the graded ring of operations $\prod_{n \geqslant 1} \qq [c_n]$
with $c_n$ of degree $-2$. This ring has trivial differential.

The group $W$ acts on this graded ring (via ring homomorphisms),
it is defined by $w c_n =-c_n$.
We thus have a map of graded rings (without $W$--action)
\[
w \co \ocal_\fcal \to \ocal_\fcal
\]
and a change of rings functor
$w^*$ from $\ocal_\fcal \leftmod$ to itself.
For a module $N$, $w^*N$ is the same underlying
set, but now each $c_n$ acts as $-c_n$.
\end{definition}

We use the notation
\[
\ecal^{-1} \ocal_\fcal = \colim_{n \geqslant 1}  \ocal_\fcal[c_1^{-1}, \dots, c_n^{-1}].
\]
It is easy to see that $\ecal^{-1} \ocal_\fcal$ is also a graded ring with $W$--action.
As a vector space, $(\ecal^{-1} \ocal_\fcal)_{2k}$ is
$\prod_{a \geqslant 1} \qq $
for $k \leqslant 0$ and is $\oplus_{a \geqslant 1} \qq$ for
$k >0$.
For any $\ocal_\fcal$--module $N$, we define $\ecal^{-1} N$ to be
$\ecal^{-1} \ocal_\fcal \otimes_{\ocal_\fcal} N$. The tensor product
has the diagonal action of $W$.

\begin{definition}\label{def:toralmodel}
An object $A= (\beta \co N \to \ecal^{-1} \ocal_\fcal \otimes U)$
of $d\acal(\ccal)$ consists of the following data:
\begin{itemize}[noitemsep]
\item an $\ocal_\fcal$--module $N$ in the category $\ch(\qq[W])$,
\item an object $U$ of $\ch(\qq[W])$,
\item a map $\beta$ of $\ocal_\fcal$--modules in the category $\ch(\qq[W])$,
\item with the requirement that
$\ecal^{-1} \beta$ is an isomorphism.
\end{itemize}

Let $B= (\beta' \co N' \to \ecal^{-1} \ocal_\fcal \otimes U')$
be another object of $d\acal(\ccal)$.
A map $(\theta, \phi) \co A \to B$ in this category consists of the following data:
\begin{itemize}[noitemsep]
\item a map $\theta \co N \to N'$ of $\ocal_\fcal$--modules
in the category $\ch(\qq[W])$
\item a map $\phi \co U \to U'$
in the category $\ch(\qq[W])$
\item with the requirement that the obvious square
involving the structure maps commutes.
\end{itemize}

We call $d \acal(\ccal)$ the \textbf{algebraic
model for toral spectra}.
The subcategory of objects with zero differentials in all places is called
$\acal(\ccal)$,  the \textbf{abelian model for toral spectra}.
\end{definition}

We let $S^0 =  \left( \ocal_\fcal \to \ecal^{-1} \ocal_\fcal \otimes \qq \right)$
where $\qq$ has trivial $W$--action.
This is the unit of a monoidal product on $d \acal(\ccal)$
(although we will make no direct use of that in this paper).

The abelian and algebraic models for rational $\torus$--equivariant spectra
have similar descriptions.

\begin{definition}
The category $d\acal(\torus)$ is defined as in
Definition \ref{def:toralmodel} but using $\ch(\qq)$
instead of $\ch(\qq[W])$. We call this
the \textbf{algebraic model for $\torus$--spectra}.

The full subcategory of $d \acal(\torus)$ consisting of
objects with zero differentials is called
$\acal(\torus)$, the \textbf{abelian model for $\torus$--spectra}.
\end{definition}

There is an obvious forgetful functor relating
$d\acal(\torus)$ and $d\acal(\ccal)$.
The results of Section \ref{sec:toral} will show that
this forgetful functor is the algebraic version of
the forgetful functor from $\ccal \osp$ to $\tsp$.

\begin{lemma}\label{lem:leftadjoint}
There is an adjoint pair relating
$d\acal(\torus)$ and $d\acal(\ccal)$. The left adjoint $\mathbb{D}$
takes
\[
\beta \co N \to \ecal^{-1} \ocal_\fcal \otimes U
\]
in $\acal$ to the following composite
\[
N \oplus w^* N
\overset{\beta \oplus w^* \beta}{\longrightarrow}
(\ecal^{-1} \ocal_\fcal \otimes U) \oplus
(w^* \ecal^{-1} \ocal_\fcal \otimes U)
\overset{\id \oplus w}{\longrightarrow}
\ecal^{-1} \ocal_\fcal \otimes U \oplus
\ecal^{-1} \ocal_\fcal \otimes U
\]
The $W$--action then simply swaps the two summands.
The right adjoint $i^*$  is
the forgetful functor from $d\acal(\ccal)$ to $d\acal(\torus)$.
\end{lemma}

\begin{rmk}\label{rmk:behaviour}
In \cite[Corollary 3.2]{greo2} the algebraic model for the toral part is described
as $\acal(\torus)$ with a $W$--action. This is not sufficiently precise, the more correct
statement is that the algebraic model for the toral part is $\acal(\ccal)$
as in Definition \ref{def:toralmodel}.
See Lemma \ref{lem:homotopycalc} for the calculation which shows how
the given action of $W$ on $\ocal_\fcal$ is obtained from topological data.

The problem with the proof of the corollary is as follows,
let $Z$ be a $\torus$--spectrum (or a $\torus$--space)
with a $\torus$-equivariant map of order two $z \co Z \to Z$.
Choose a reflection $\hat{w} \in O(2)$ and let
$R_{\hat{w}} \co O(2)_+ \to O(2)_+$ be right multiplication by $\hat{w}$.
Then the map
\[
R_{\hat{w}} \smashprod z \co O(2)_+ \smashprod_{\torus} Z \longrightarrow O(2)_+ \smashprod_{\torus} Z
\]
is not well defined due to the equalisation of the $\torus$--actions.
This is easily seen at the space level, where
$R_{\hat{w}} \smashprod z ( [\sigma, x] ) = [\sigma \hat{w}, z(x)]$.

Instead, let $j \co \torus \to \torus$ be the inversion map
and $j^*$ be the change of groups functor. If
$z$ is a map from $Z$ to $j^* Z$,
such that $j^* z \circ z$ is the identity,
then the map $R_{\hat{w}} \smashprod z$ is well defined.
Obviously, if $Y$ is an $O(2)$--spectrum then $i^*Y$ has such a map,
given by the action of $\hat{w}$.
With this interpretation of $W$--actions,
\cite[Proposition 3.1 and Corollary 3.2]{greo2} are correct.
This idea of skewed actions on a category is considered in more
detail in \cite[Chapters 7 and 8]{barnesthesis}.
\end{rmk}

\subsection{Model structures on
\texorpdfstring{$d\acal(\ccal)$}{dA(C)}}\label{subsec:toralmodel}

We use the dualisable model structure
on $d \acal(\torus)$, see \cite[Theorem 6.6]{barnesmonoidal}.
This is a proper monoidal
model category that satisfies the monoid axiom
and whose weak equivalences are the homology isomorphisms.
We can lift this model structure to the toral case using
the lifting lemma \cite[Theorem 11.3.2]{hir03}
and the adjoint pair $(\mathbb{D}, i^*)$.

\begin{theorem}
There is a model structure on $d \acal(\ccal)$ where
the weak equivalences are the homology isomorphisms
and the fibrations are those maps which forget to
fibrations in the dualisable model structure on
$d \acal(\torus)$.
This model structure is proper, cofibrantly generated,
monoidal and satisfies the monoid axiom.
The generating cofibrations and acyclic cofibrations
are given by applying
$\mathbb{D}$ to the generating sets for the
dualisable model structure on
$d \acal(\torus)$.
\end{theorem}

Note that the ring map $w \co \ocal_\fcal \to \ocal_\fcal$
induces a map $w \co S^0 \to w^* S^0$.
It follows that $S^0$ (with $W$--action) is a retract of
$\mathbb{D} S^0$ and hence is cofibrant in
$d \acal(\ccal)$.

In \cite{gre99}, Greenlees constructs a functor $\pi_*^{\acal}$
from the homotopy category of rational $\torus$--spectra
to $\acal(\torus)$.
For a rational $\torus$--spectrum $X$, let $\pi^\acal_*(X)$
be the following object of $\acal$.
For details of the spectra $DE \fcal_+$ and $\widetilde{E} \fcal$
see Definition \ref{def:universalspaces}. The spectrum $\Phi^{\torus} X$
is the geometric $\torus$--fixed points of $X$.
\[
\pi^\acal_*(X) = \big(
\pi_*^{\torus} (X \smashprod DE \fcal_+)
\longrightarrow
\pi_*^{\torus} (X \smashprod DE \fcal_+ \smashprod \widetilde{E} \fcal)
\cong \ecal^{-1} \ocal_\fcal \otimes \pi_*(\Phi^{\torus} X) \big)
\]
Since $DE\fcal_+$ and $\widetilde{E} \fcal$
can be constructed as $O(2)$--spectra, we can
extend $\pi^\acal_*$ to the toral case
by keeping track of the $W=O(2)/\torus$ action on
$\pi_*^{\torus} (X)$ for $X$ an $O(2)$--spectrum.
This functor fits into an Adams short exact sequence for the toral part, the proof
follows the same pattern as \cite[Theorem 5.6.6]{gre99}.

\begin{theorem}\label{thm:toraladams}
For $X$ and $Y$ rational $O(2)$--spectra,
there is an Adams short exact sequence as below,
where $[-,-]^{\ccal}_*$ denotes maps in the homotopy
category of toral spectra.
\[
0 \to
\ext_{\acal(\ccal)} (\pi_*^{\acal}(\Sigma X),\pi_*^{\acal}( Y))
\to
[X,Y]^{\ccal}_*
\to
\hom_{\acal(\ccal)} (\pi_*^{\acal}(X),\pi_*^{\acal}(Y))
\to 0
\]
\end{theorem}

\section{Toral spectra}\label{sec:toral}

In this section we show that the model category of
toral $O(2)$--spectra,  $\ccal \osp$,
is Quillen equivalent to the category $d\acal(\ccal)$.
The method is an extension of \cite{BGKS}.

The first step is to separate that part of toral $O(2)$--equivariant
homotopy theory that is determined by the finite cyclic subgroups
from that determined by $\torus$.
Proposition \ref{prop:Smodulesequivalence} gives this separation,
see \cite[Proposition 3.2.5]{BGKS} for the $\torus$--equivariant analogue
and further explanation of the underlying idea.

The next step is to take $\torus$--fixed points, so that we are now working
with non-equivariant spectra. This removal of equivariance is achieved in
Corollary \ref{cor:removeequivariancecell},
which is a generalisation of \cite[Corollary 3.3.6]{BGKS}.
The major difference is that the
$\torus$--fixed points of an $O(2)$--spectrum
define a spectrum with a $W$--action, see Lemma \ref{lem:fixed}.

The third step is to replace categories based on rational spectra (with an action of $W$)
with categories based on chain complexes (with an action of $W$) using \cite{shiHZ},
see Theorem \ref{thm:movetoalgebra}.
The remaining steps are analogues of \cite[Section 4]{BGKS},
where we complete our series of Quillen equivalences
with the model $d \acal (\ccal)$ by removing the localisations and
cellularisations in our constructions.
See Propositions \ref{prop:simplemodel} and Theorem \ref{thm:gammaequiv}.

\subsection{Isotropy separation}\label{subsec:separate}

We briefly recap the notion of a diagram of model categories
and the category of generalised diagrams.
We let $\pscr$ denote the pullback category
$\bullet \to \bullet \leftarrow \bullet$.

\begin{definition}
A \textbf{$\pscr$--diagram of model categories}
$R^\bullet$ is a pair of Quillen pairs
\[
\begin{array}{rcl}
L:
\acal
&
\adjunct
&
\bcal : R \\
F:
\ccal
&
\adjunct
&
\bcal : G \\
\end{array}
\]
with $L$ and $F$ the left adjoints.
We will usually draw this as the diagram below.
\[
\xymatrix{
\acal
\ar@<+1ex>[r]^L
&
\bcal
\ar@<+0.5ex>[l]^R
\ar@<-0.5ex>[r]_G
&
\ccal
\ar@<-1ex>[l]_F
}
\]
\end{definition}

We can then define the category of generalised diagrams in $R^\bullet$.
This is sometimes also called the category of sections.
\begin{definition}
Given a $\pscr$--diagram of model categories $R^\bullet$ as above,
we can define a new category, $R^\bullet \leftmod$.
The objects are pairs of morphisms in $\bcal$:
$\alpha \co La \to b$ and $\gamma \co Fc \to b$.
We write such an object as $(a,\alpha, b, \gamma, c)$.
A morphism in $R^\bullet \leftmod$ from
$(a,\alpha, b, \gamma, c)$ to $(a',\alpha', b', \gamma', c')$
is a triple of maps
$x \co a \to a'$ in $\acal$,
$y \co b \to b'$ in $\bcal$,
$z \co c \to c'$ in $\ccal$ such that we have a commuting diagram in $\bcal$
\[
\xymatrix{
La \ar[r]^\alpha
\ar[d]^{Lx}
& b
\ar[d]^y
& Fc \ar[l]_\gamma
\ar[d]^{Fz} \\
La' \ar[r]^{\alpha'}
& b'
& Fc' \ar[l]_{\gamma'}
}
\]
\end{definition}

If each category in the diagram $R^\bullet$ is proper and cellular, then
the category $R^\bullet \leftmod$ has a proper and cellular model structure with weak equivalences
and cofibrations defined objectwise by \cite[Proposition 3.3]{gsmodules}.

We can separate the homotopical information of toral $O(2)$--spectra
into three parts. The first part takes care of the homotopical information coming from
the finite cyclic subgroups. The second
part deals with the homotopical information coming from $\torus$.
The third part is a comparison term.
We have already removed the behaviour of the dihedral groups
in Theorem \ref{thm:mainsplitting}. The first step is to carefully construct a
commutative ring spectrum with some special properties.

\begin{definition}\label{def:universalspaces}
Let $\fcal$ be the collection of finite cyclic subgroups of $O(2)$.
There is a universal space for this family called $E \fcal$
where $E \fcal^H$ is non-equivariantly contractible
for each finite cyclic subgroup $H$
and $E \fcal^{K} = \emptyset$ for all other subgroups $K$.
This is an $O(2)$--CW--complex and is built from cells of the form
$O(2)/K_+$ for $K \in \fcal$.
We define $\widetilde{E} \fcal$ via the cofibre sequence of
$O(2)$--spaces,
\[
E \fcal_+ \to S^0 \to \widetilde{E} \fcal.
\]
We define $D E \fcal_+$ to be $F(E \fcal_+, N^\# S)$,
where $N^\#$ is the lax monoidal right adjoint
described in \cite[Theorem IV.3.9]{EKMM97}
from $O(2)$--equivariant EKMM $S$--modules to $O(2) \Sp$.
\end{definition}

\begin{lemma}
The spectrum $D E \fcal_+$ is a commutative ring spectrum
that is fibrant in the category of
toral $O(2)$--equivariant orthogonal spectra. This forgets
to the commutative ring $\torus$--spectrum $D E \fcal_+$ constructed
in \cite[Definition 3.2.2]{BGKS}.
\end{lemma}
\begin{proof}
The ring spectrum constructed in \cite{BGKS} is made using the
same process as above, but starting in $\torus$--equivariant $S$--modules.
Hence the statement about $i^* DE \fcal_+$ follows immediately.
The spectrum $DE \fcal_+$ is fibrant in $O(2) \Sp$, since
$N^\# S$ is fibrant. Hence $DE \fcal_+$ is fibrant in $\ccal O(2) \Sp$
(recall the identity functor $\ccal O(2) \Sp \to O(2) \Sp$ is a left Quillen functor).
\end{proof}

Note that we do not require $DE \fcal_+$ to be fibrant in $\ccal \osp$:
we do not need it to be rational, only that its $\torus$--fixed points
are weakly equivalent to its derived $\torus$--fixed points.
That is, we need
\[
\pi_*(DE \fcal_+^\torus) \cong \pi_*^\torus(DE \fcal_+)
\]
which holds as $DE \fcal_+$  is fibrant in $O(2) \Sp$.

From this ring spectrum we can make three model categories.
In the following, whenever we have a ring object $A$ in a model category
$\mcal$, we will equip $A \leftmod$ with the lifted model structure,
where fibrations and weak equivalences are defined by
forgetting to $\mcal$.
\begin{itemize}
\item $DE \fcal_+ \leftmod$, the category of $DE \fcal_+$--modules
in $\ccal \osp$.
This model category captures the information coming from the finite cyclic groups.
\item $L_{\widetilde{E} \fcal} \ccal \osp$,
the model category $\ccal \osp$, localised
at the homology theory $\widetilde{E} \fcal$.
This model category captures the information coming from  $\torus$.
\item $L_{\widetilde{E} \fcal \smashprod DE \fcal_+} DE \fcal_+ \leftmod$
the category of $DE \fcal_+$--modules
in $\ccal \osp$, localised at the
homology theory $\widetilde{E} \fcal \smashprod DE \fcal_+$.
This model category captures the interaction of the first two.
\end{itemize}

Now we can give our diagram of model categories that
separates the behaviour of the finite cyclic groups from the rest.

\begin{definition}
We define $S^\bullet$ to be the following diagram of model categories.
\[
\xymatrix@C+1cm{
DE \fcal_+ \leftmod
\ar@<+1ex>[r]^-{\id}
&
L_{\widetilde{E} \fcal \smashprod DE \fcal_+} DE \fcal_+ \leftmod
\ar@<+0.5ex>[l]^-{\id}
\ar@<-0.5ex>[r]_-{U}
&
L_{\widetilde{E}\fcal } \ccal \osp
\ar@<-1ex>[l]_-{DE \fcal_+ \smashprod -}
}
\]
We thus have a cellular
model category $S^\bullet \leftmod$, that is both proper and stable.
\end{definition}

Given any $O(2)$--spectrum $X$, we have an $S^\bullet$--module
\[
S^\bullet \smashprod X
: =
(DE \fcal_+ \smashprod X, \id, DE \fcal_+ \smashprod X, \id , X).
\]
The functor $S^\bullet \smashprod -$ has a right adjoint called $\pullback$,
which is constructed just after \cite[Definition 3.2.3]{BGKS}.
This right adjoint sends an object $(A,\alpha, B, \gamma, C)$
to the pullback of the diagram
\[
\xymatrix@C+0.3cm{
A \ar[r]^{\alpha} & B &
\ar[l]_-\gamma C \smashprod DE \fcal_+ & \ar[l]_-{C \smashprod \textrm{unit}} C
}.
\]
Similarly to \cite[Proposition 3.2.4]{BGKS} we have a Quillen pair as below.
\[
S^\bullet \smashprod - :
\ccal \osp
\adjunct
S^\bullet \leftmod
: \pullback
\]

Now we want to relate this to the $\torus$--equivariant separation.
The diagram $S^\bullet$ is a diagram of model categories of $O(2)$--spectra (with extra structure),
so at each vertex we can make an analogous category built from $\torus$--equivariant spectra.

\begin{definition}\label{def:isbullet}
Let $i^* D E \fcal_+ \leftmod$ denote modules over $i^* DE \fcal_+$ in
$\tsp$. There is a diagram of
model categories $i^* S^\bullet$,
made from the three model categories:
\[
i^* D E \fcal_+ \leftmod,
\quad
L_{i^* \widetilde{E} \fcal \smashprod i^* DE \fcal_+} i^* DE \fcal_+ \leftmod
\quad \textrm{and} \quad
L_{i^* \widetilde{E} \fcal} \tsp
\]
\end{definition}
The above diagram of model categories is precisely the
diagram of \cite[Definition 3.2.3]{BGKS}
and we have a square of Quillen functors as below.
\[
\xymatrix@C+1.5cm@R+0cm{
\ccal O(2) \Sp_\qq
\ar@<0.1cm>[r]^{S^\bullet \smashprod -}
\ar@<0cm>[d]^{i^*} &
S^\bullet \leftmod
\ar@<0cm>[d]^{i^*}
\ar@<0.1cm>[l]^{\pullback} \\
\torus \Sp_\qq
\ar@<0.1cm>[r]^{i^* S^\bullet \smashprod -} &
i^* S^\bullet \leftmod
\ar@<0.1cm>[l]^{\pullback}
}
\]

\begin{lemma}\label{lem:diagramcommute}
The forgetful functors $i^*$ commute with both
the horizontal left adjoints and the horizontal right adjoints.
They also preserve fibrations and cofibrations.
A map $(x,y,z)$ in $S^\bullet \leftmod$ is a weak equivalence
if and only if $i^*(x,y,z) = (i^*x, i^*y, i^* z)$ is a
weak equivalence in $i^* S^\bullet \leftmod$.
\end{lemma}
\begin{proof}
That $i^*$ commutes with the horizontal functors is immediate from the
definitions. The left hand $i^*$ preserves fibrations and cofibrations
by Proposition \ref{prop:commforgetful}.

The right hand $i^*$ preserves and detects weak equivalences in
$S^\bullet \leftmod$ as weak equivalences are detected objectwise and
$i^* \co \ccal \osp \to \tsp$ preserves and detects weak equivalences.
Cofibrations are also defined objectwise, so $i^*$ preserves cofibrations
as it does so on each component model category of $S^\bullet$.

On each component model category of the right hand side, $i^*$ preserves fibrations.
Since the fibrations in $S^\bullet \leftmod$ are defined
in terms of certain pullbacks (which are preserved by $i^*$)
it follows that the right hand $i^*$ also preserves fibrations.
\end{proof}

To turn the horizontal adjunctions of this square into Quillen equivalences,
we apply the Cellularization Principle of \cite[Proposition 2.7]{gscell}.
This result gives conditions under which a Quillen adjunction becomes a
Quillen equivalence after cellularising (right Bousfield localising)
both sides of the adjunction.

Let $e_{C_n}$ denote the idempotent in the rationalised Burnside ring for $C_n$
corresponding to $C_n$ and let $\sphspec_{C_n}$ denote the $C_n$--equivariant sphere spectrum.
The generators of $\ccal \osp$ are
the spectra $O(2)_+ \smashprod_{C_n} e_{C_n} \sphspec_{C_n}$ for $n \geqslant 1$
and the cofibrant replacement of the $O(2)$--equivariant sphere in $\ccal \osp$:
$E \ccal_+ \smashprod \sphspec$.
Let $K_{\Top}$ be the set of images
of these objects under the functor $S^\bullet \smashprod -$, and all
(integer) suspensions and desuspensions thereof.
The elements of this set will be called \textbf{cells} and we will cellularise
(right Bousfield localise) $S^\bullet \leftmod$ at this set.

To apply the cellularisation principle, we need to know that
the cells $K_{\Top}$ are homotopically compact
(also known as small or compact) in the sense of
\cite[Definition 2.1.2]{ss03stabmodcat}.
The arguments of \cite[Section 3.2]{BGKS}
apply verbatim, so we leave the details to that reference.

\begin{proposition}\label{prop:Smodulesequivalence}
There is a Quillen equivalence
\[
S^\bullet \smashprod - :
\ccal \osp
\adjunct
K_{\Top} \cell S^\bullet \leftmod
: \pullback
\]
\end{proposition}
\begin{proof}
This result follows from the same proof as for the $\torus$--case.
By \cite[Proposition 2.7]{gscell} it suffices to show that the
derived unit is a weak equivalence on the generators of $\osp$.
The pullback of a fibrant replacement of $S^\bullet \smashprod X$
is given by the homotopy pullback of the following diagram of O(2)--spectra:
\[
DE \fcal_+ \smashprod X \smashprod S^0 \qq
\longrightarrow
DE \fcal_+ \smashprod \widetilde{E} \fcal \smashprod X \smashprod S^0 \qq
\longleftarrow
\widetilde{E} \fcal \smashprod X \smashprod S^0 \qq.
\]
Since we are in a stable model category, the above is
weakly equivalent to $X \smashprod S^0 \qq$ smashed with the homotopy pullback
of
$DE \fcal_+
\longrightarrow
DE \fcal_+ \smashprod \widetilde{E} \fcal
\longleftarrow
\widetilde{E} \fcal$
which is $\sphspec$. The derived
unit map is induced by the unit map
$X \to DE \fcal_+ \smashprod X$ and
$S^0 \to \widetilde{E} \fcal$
and hence is a rational equivalence.
\end{proof}

We now wish to compare with the $\tsp$ version.
As above, the generators of $\tsp$
are the $\torus$--equivariant sphere spectrum $\sphspec$ and
the spectra $\torus_+ \smashprod_{C_n} e_{C_n} \sphspec_{C_n}$ for $n \geqslant 1$.
Let $i^* K_{\Top}$ be the set of images
of these objects under the functor $i^* S^\bullet \smashprod -$, and all
(integer) suspensions and desuspensions thereof.

\begin{lemma}
The functors $i^*$ below are right Quillen functors that
commute with both the horizontal left adjoints and the horizontal right adjoints.
Furthermore the functors $i^*$ preserve and detect all weak equivalences.
\[
\xymatrix@C+2cm@R+0cm{
\ccal O(2) \Sp_\qq
\ar@<0.1cm>[r]^(0.45){S^\bullet \smashprod -}
\ar@<0cm>[d]^{i^*} &
K_{\Top} \cell S^\bullet \leftmod
\ar@<0cm>[d]^{i^*}
\ar@<0.1cm>[l]^(0.55){\pullback} \\
\torus \Sp_\qq
\ar@<0.1cm>[r]^(0.45){i^* S^\bullet \smashprod -} &
i^* K_{\Top} \cell i^* S^\bullet \leftmod
\ar@<0.1cm>[l]^(0.55){\pullback}
}
\]
\end{lemma}
\begin{proof}
The commutativity follows immediately from the definitions.
Let $\rr \pullback $ denote the right derived functor
of $\pullback$.
Then $f \co X \to Y$ in $K_{\Top} \cell S^\bullet \leftmod$
is a weak equivalence if and only if
$\rr \pullback f$ is a weak equivalence of $\ccal \osp$.
This holds if and only if $i^* \rr \pullback f = \rr \pullback i^* f$
is a weak equivalence in $\tsp$ by Proposition \ref{prop:commforgetful}.
Finally, that holds if and only if
$i^* f$ is a weak equivalence in
$i^* K_{\Top} \cell i^* S^\bullet \leftmod$, as the lower adjunction is a
Quillen equivalence.
Thus the right hand $i^*$ is a right Quillen functor after cellularisation
as the fibrations are unchanged and it preserves weak equivalences.
\end{proof}

Thus we have separated the homotopical information of $S_\ccal \leftmod$ into a diagram of three
model categories. The advantage of doing so is that we may now remove the equivariance
from the model category whilst keeping the correct homotopy category.

\subsection{Removing equivariance}\label{subsec:fixedpoints}

We start with the general pattern used in this section.
Let $A$ be a (commutative) ring spectrum in $O(2)$--equivariant spectra.
Then $A^{\torus}$ is a (commutative) ring object in $\Sp[W]$ and there is a map
of commutative rings $a \co \varepsilon^* A^{\torus} \to A$.
Using \cite[Section 4]{gsfixed}
we get a Quillen adjunction
\[
a_\sharp = A \smashprod_{\varepsilon^* A^{\torus}} \varepsilon^* (-):
A^{\torus} \leftmod
\adjunct
A \leftmod
: (-)^{\torus}.
\]
That reference has a number of examples where this
kind of adjunction is a Quillen equivalence.
We want to use this type of adjunction to remove $\torus$--equivariance from
$S^\bullet \leftmod$. We do so by considering each of the three component
model categories in turn. In the first we consider $A=DE \fcal_+$.

\begin{lemma}\label{lem:torusfixed}
The adjunction below is a Quillen equivalence.
\[
a_\sharp :
DE\fcal_+^{\torus} \leftmod
\adjunct
DE\fcal_+ \leftmod
: (-)^{\torus}
\]
\end{lemma}
\begin{proof}
The forgetful functors, to $\wsp$ on the left and
$\tsp$ on the right, commute with both
$a_\sharp$ and $(-)^{\torus}$, just as in Proposition \ref{prop:commforgetful}.
Furthermore these forgetful functors preserve fibrant objects and cofibrant objects
and preserve and detect weak equivalences.
We also know that the analogous adjunction at the level of $\torus$--equivariant
spectra is a Quillen equivalence by \cite[Proposition 3.3.1]{BGKS}.

Let $X \in DE\fcal_+^{\torus} \leftmod$ be cofibrant and
$Y \in DE\fcal_+ \leftmod$ be fibrant.
The map $f \co a_\sharp X \to Y$ is a weak equivalence
if and only if the adjoint map
$\widehat{f} \co X \to Y^\torus$ is a weak equivalence,
as this holds for $i^*f$ and $i^* (\widehat{f}) = \widehat{i^* f}$.
\end{proof}

The next step is to repeat the above with an additional Bousfield localisation.
The model category $L_{\widetilde{E} \fcal \smashprod DE \fcal_+} DE \fcal_+ \leftmod$
can be described as the localisation of the model category $DE \fcal_+ \leftmod$
at the set of maps $\Sigma^* f$:
the set of all suspensions and desuspensions of
$f : DE\fcal_+ \to DE\fcal_+ \smashprod \tilde{E}\fcal$.
Let $(\Sigma^* f)^\torus$: be the set of maps obtained by applying the derived right adjoint to the maps
$\Sigma^* f$. By \cite[Theorem 3.3.20, part 1b]{hir03} we obtain a Quillen equivalence.

\begin{lemma}
The adjunction below is a Quillen equivalence.
\[
a_\sharp :
L_{(\Sigma^* f)^\torus}
DE\fcal_+^\torus \leftmod
\adjunct
L_{\Sigma^* f} DE \fcal_+ \leftmod
: (-)^{\torus}.
\]
\end{lemma}

The final version is to use $A=S^0$, in which case
the left adjoint $a_\sharp$ is simply $\varepsilon^*(-)$.

\begin{lemma}
The adjunction below is a symmetric monoidal Quillen equivalence.
\[
\varepsilon_* :
\wsp
\adjunct
L_{\widetilde{E} \fcal} \ccal \osp
: (-)^{\torus}
\]
\end{lemma}
\begin{proof}
This follows by the same arguments as for
\cite[Proposition 3.3.3]{BGKS}, namely that
the derived right adjoint behaves
as the geometric $\torus$--fixed point functor,
the left hand side is generated by $W_+$
and the right hand side is generated by $O(2)/\torus_+$.
\end{proof}

We can extend the functor $(-)^{\torus}$ to the level of
diagrams of model categories.

\begin{definition}
We define $S^\bullet_{\Top}$ to be the diagram of model categories below.
Here, $DE\fcal_+^\torus \leftmod$
means modules over $DE\fcal_+^\torus$
in the model category $\wsp$.
\[
\xymatrix@C+1cm{
DE\fcal_+^\torus \leftmod
\ar@<+1ex>[r]^-{\id }
&
L_{(\Sigma^* f)^\torus}
DE\fcal_+^\torus \leftmod
\ar@<+0.5ex>[l]^-{\id}
\ar@<-0.5ex>[r]_-{}
&
\wsp
\ar@<-1ex>[l]_-{DE \fcal_+ \smashprod - }
}
\]
The unmarked functor is simply the forgetful functor.
\end{definition}

Because of the way we have constructed $S^\bullet_{\Top}$
it follows that $(a_\sharp, (-)^\torus)$
gives a map of diagrams of model categories from
$S^\bullet_{\Top}$ to $S^\bullet$. Since each of the components is
a Quillen equivalence, we immediately obtain the following.

\begin{theorem}\label{thm:removeequivariance}
There is a Quillen equivalence.
\[\xymatrix@R+0.5cm@C+2cm{
S^\bullet_{\Top} \leftmod
\ar@<0.1cm>[r]^{a_\sharp}
&
S^\bullet \leftmod
\ar@<0.1cm>[l]^{(-)^{\torus}}
}\]
\end{theorem}

We define $K_{\Top}^{\torus}$, the set of cells for $S^\bullet_{\Top} \leftmod$,
to be the set of objects given by
applying the derived functor of $(-)^{\torus}$ to $K_{\Top}$.
By \cite[Corollary 2.8]{gscell}
we see that the Quillen equivalence
above is preserved by cellularisation.
\begin{corollary}\label{cor:removeequivariancecell}
The adjunction below is a Quillen equivalence.
\[\xymatrix@R+0.5cm@C+2cm{
K_{\Top}^{\torus} \cell S^\bullet_{\Top} \leftmod
\ar@<0.1cm>[r]^{a_\sharp}
&
K_{\Top} \cell S^\bullet \leftmod
\ar@<0.1cm>[l]^{(-)^{\torus}}
}\]
\end{corollary}

The forgetful functor $i^*$ relates the toral spectra version
and the $\torus$--equivariant version of this result \cite{BGKS}.
Recall that we have the diagram of
model categories based on $\torus$--spectra which we call
$i^* S^\bullet$, see Definition \ref{def:isbullet}.
Just as we constructed $S^\bullet_{\Top}$, we
can also make
a diagram of model categories $i^* S^\bullet_{\Top}$
that is the `fixed points' of $i^* S^\bullet$.
We then have the following result which is similar in nature
to Proposition \ref{prop:commforgetful}.

\begin{corollary}
There is a diagram of Quillen functors as below.
We note that $i^*$ commutes with both $a_\sharp$ and $(-)^{\torus}$
up to natural isomorphism.
\[\xymatrix@R+0cm@C+2cm{
S^\bullet_{\Top} \leftmod
\ar@<0.1cm>[r]^{a_\sharp}
\ar@<0.1cm>[d]^{i^*}
&
S^\bullet \leftmod
\ar@<0.1cm>[l]^{(-)^{\torus}}
\ar@<0.1cm>[d]^{i^*}
\\
i^* S^\bullet_{\Top} \leftmod
\ar@<0.1cm>[r]^{a_\sharp}
&
i^* S^\bullet \leftmod
\ar@<0.1cm>[l]^{(-)^{\torus}}
}\]
\end{corollary}

\subsection{Moving to algebra}\label{subsec:toalgebra}

We want to replace the model category
$K_{\Top}^{\torus} \cell S_{\Top}^\bullet \leftmod$
by a Quillen equivalent $\ch(\qq[W])$--model category.
The method is the same as in \cite[Section 3.4]{BGKS}, we briefly describe the process.
We forget structure to get a diagram of model categories based on
symmetric spectra with $W$--action (rather than orthogonal spectra).
Then we `smash' with $\h \qq$ to get a diagram of model categories based on
$\h \qq$--modules in symmetric spectra with $W$--action.
Then one can apply \cite[Lemma 5.7]{kedziorekexceptional} (an extension of \cite{shiHZ})
to get a diagram of model categories based on $\ch(\qq[W])$.
We leave the fine details to the references
and simply state the consequences.

\begin{theorem}\label{thm:movetoalgebra}
There exists a commutative ring $S_t$ in $\ch(\qq[W])$
and a set of maps of $S_t$--modules $A$ such that
\begin{itemize}[noitemsep]
  \item there is an isomorphism $\h_*(S_t) \cong \pi_*^\torus (DE \fcal_+)$
  \item for any $a \in A$, there is a canonical $g \in (\Sigma^* f)^\torus$ with
        $\h_*(a) \cong \pi_*(g)$.
  \item there is a zig-zag of Quillen equivalences
        between $S^\bullet_{\Top} \leftmod$ and the diagram of model categories
        $S^\bullet_{t} \leftmod$
\[
\xymatrix@C+1cm{
S_t \leftmod
\ar@<+1ex>[r]^-{\id}
&
L_{A} S_t \leftmod
\ar@<+0.5ex>[l]^-{\id}
\ar@<-0.5ex>[r]_-{U}
&
\ch(\qq[W])
\ar@<-1ex>[l]_-{S_t \smashprod -}
}.
\]
\end{itemize}
\end{theorem}

As in \cite[Section 3.4]{BGKS},
the zig-zag between $S^\bullet_{\Top} \leftmod$ and
$S^\bullet_{t} \leftmod$ actually
consists of objectwise Quillen equivalences.
In particular, there is a zig-zag of Quillen equivalences between
$S_t \leftmod$ and $DE\fcal_+^\torus \leftmod$ and $A$ is the
set of images of the maps in $(\Sigma^* f)^\torus$
under the derived zig-zag.

Since cellularisation is compatible with Quillen
equivalences \cite[Corollary 2.8]{gscell}, there is a set of cells $K_t$ in $S_t \leftmod$
which gives the following result. These cells $K_t$ are the images of the objects in $K_{\Top}$
under the derived zig-zag of Quillen equivalences.

\begin{corollary}
There is a zig-zag of Quillen equivalences between
the model category
$K_{\Top}^{\torus} \cell R_{\Top}^\bullet  \leftmod$
and the model category
$K_{t} \cell S_t^\bullet \leftmod$.

For each $\tau \in K_t$ there is a canonical $\sigma \in K_{\Top}$
with $\h_*(\tau) \cong \pi_*(\sigma)$.
\end{corollary}

This process is compatible with the forgetful functor,
in the sense that at each stage of the zig-zag, there is
a commutative square of functors similar to that of
Proposition \ref{prop:commforgetful}.
In particular, there is a zig-zag of Quillen equivalences between the model categories
\[
i^* K_{\Top}^{\torus} \cell i^* S_{\Top}^\bullet \leftmod
\quad \textrm{and} \quad
i^* K_{t} \cell i^* S_t^\bullet \leftmod
\]
and the forgetful functor
\[
i^* \co
K_{t} \cell S_t^\bullet \leftmod
\longrightarrow
i^* K_{t} \cell i^* S_t^\bullet \leftmod
\]
preserves and detects weak equivalences and fibrations.

\subsection{Simplifying the algebra}\label{subsec:ramod}

We can simplify the diagram $S_t$ in two ways.
First by removing the localisation at $A$, secondly by
replacing the commutative dga $S_t$ by a simpler commutative dga.
The key to both is formality arguments, we use the fact that the homology
of $S_t$ or the homology of the maps in $A$ is sufficiently well-structured to determine their
homology type.
The method is an extension of \cite[Section 4]{BGKS},
where we include
the action of $W=O(2)/\torus$ on $\ocal_\fcal$.
This extension is possible since in the
previous sections we have shown that at each our stage forgetful functors
restrict to the categories and objects used in \cite{BGKS}.
Let $\ocal_\fcal$ be the graded ring $\prod_{n \geqslant 0} \qq[c_n]$
from Section \ref{subsec:the model}. Recall that
each $c_n$ has degree $-2$.

\begin{lemma}\label{lem:homotopycalc}
There is a zig-zag of quasi-isomorphisms of commutative ring objects in
$\ch(\qq[W])$ between $S_t$ and $\ocal_\fcal$.
\end{lemma}
\begin{proof}
We already have isomorphisms of commutative ring objects in
$\ch(\qq[W])$
\[
\h_*(S_t) \cong \pi_* (DE\fcal_+^\torus)
\cong
\pi_*^{\torus} (DE\fcal_+)
=
[S^0 , DE\fcal_+]_*^{\torus}
\cong
[E\fcal_+, S^0]_*^{\torus} .
\]
By \cite[Theorems 7.4 and 7.5]{gretnq1}, in particular the first line of the
proof of Theorem 7.5, it follows that
\[
[E\fcal_+, S^0]_*^{\torus}
\cong
[E\fcal_+, E\fcal_+]_*^{\torus}
\cong
\prod_{H \in \fcal} \h^* (\classify (\torus/H)) =
\prod_{H \in \fcal} \qq[c]
= \ocal_\fcal.
\]
The $O(2)$--action on $E\fcal_+$ induces a $W$--action on
$\pi_*^{\torus} (DE\fcal_+)$ and hence there is a $W$--action on
$\h^* (\cc P^\infty) = \h^* (\classify \torus/H) = \qq[c]$.
This action sends $c^i$ to $(-1)^i c^i$ as it is induced by the
self-map of $\torus$ given by $t \mapsto t^{-1}$, which is exactly
conjugation by a reflection of $O(2)$.
This fact is also noted immediately above Lemma 7.1 of \cite{greenso3}.
Thus we know the homology of $S_t$ as a graded ring with $W$--action.
Our next task is to show that $S_t$ is quasi-isomorphic
to its homology.

There is a cycle $x_n \in S_t$
which corresponds to the idempotent
$e_n \in \ocal_\fcal$, which projects onto
factor $n$. Since this may not be $W$--fixed,
let $y_n$ be the average of $x_n$ and $w x_n$.
We then have a quasi-isomorphism
$S_t \to \prod_{n \geqslant 1} S_t[y_n^{-1}]$.
For each $n$, pick a representative
$a_n \in S_t[y_n^{-1}]$ for the homology class $c_n$.
Now let $b_n = 1/2 (a_n - wa_n)$.
Then the map sending $c_n$ to $b_n$ gives a
$W$-equivariant quasi-isomorphism
$\qq[c_n] \to S_t[y_n^{-1}]$.
Putting these together for each $n$ gives the
other half of our zig-zag.
\end{proof}

The other simplification is formally the same as
the argument in \cite[Section 4]{BGKS}, so we leave the details to the
reference.

\begin{proposition}\label{prop:simplemodel}
There is a Quillen equivalence between
$S_t^\bullet \leftmod$,
and a diagram of model categories
$S_a^\bullet \leftmod$:
\[
\xymatrix@C+1cm{
\ocal_\fcal \leftmod
\ar@<+1ex>[r]^-{\ecal^{-1} \ocal_\fcal \otimes_{\ocal_\fcal} -}
&
\ecal^{-1} \ocal_\fcal \leftmod
\ar@<+0.5ex>[l]
\ar@<-0.5ex>[r]
&
\ch(\qq[W])
\ar@<-1ex>[l]_-{\ecal^{-1} \ocal_\fcal \otimes -}
}
\]
where both unmarked functors are forgetful functors.
Let $K_a$ be the images of the cells $K_t$ in
$S_a^\bullet \leftmod$.
Then we have Quillen equivalences between
$K_{t} \cell S_t^\bullet \leftmod$
and
$K_{a} \cell S_a^\bullet \leftmod$.
\end{proposition}

As in previous cases, application of the forgetful functor
from $\ch(\qq[W]) \to \ch(\qq)$
reduces us to the case of rational $\torus$--spectra.
The forgetful functor
\[
i^* \co
K_{a} \cell S_a^\bullet \leftmod
\longrightarrow
i^* S_{a} \cell i^* S_a^\bullet \leftmod
\]
preserves and detects weak equivalences and fibrations.

\subsection{Comparison with the algebraic model}\label{subsec:torsion}

We now turn to comparing $S_a^\bullet \leftmod$
to the algebraic model $d \acal(\ccal)$ of Section \ref{sec:toralmodel}.
We first introduce an adjoint pair relating
$S_a^\bullet \leftmod$ and
$d \acal(\ccal)$.
An object
\[
\beta \co M \to \ecal^{-1} \ocal_\fcal \otimes V
\]
of $d \acal(\ccal)$ gives
an object of $S_a^\bullet \leftmod$
defined by
\[
(M, \ecal^{-1}\beta, \ecal^{-1} \ocal_\fcal \otimes V, \id , V)
\]
This functor, which we call $l^*$, includes $d \acal(\ccal)$ into
$S_a^\bullet \leftmod$, it has a right adjoint $\Gamma$.
For more details, see \cite[Section 7]{barnesmonoidal}.

\begin{theorem}\label{thm:gammaequiv}
The pair $(l^*,\Gamma)$ induces a Quillen equivalence between
the model categories
$d \acal(\ccal)$ and
$K_a \cell S_a^\bullet \leftmod$.
Furthermore, the square of left adjoints (and the square of right adjoints)
of the diagram below commute.
\[
\xymatrix@C+1cm{
d \acal(\ccal)
\ar@<0.1cm>[r]^(0.4){l^*} \ar@<0.1cm>[d]^{i^*}
&
K_a \cell S_a^\bullet \leftmod
\ar@<0.1cm>[l]^(0.6){\Gamma} \ar@<0.1cm>[d]^{i^*}
\\
d \acal(\torus)
\ar@<0.1cm>[r]^(0.4){l^*} \ar@<0.1cm>[u]^{\mathbb{D}}
&
i^* K_a \cell i^* S_a^\bullet \leftmod
\ar@<0.1cm>[l]^(0.6){\Gamma} \ar@<0.1cm>[u]^{\mathbb{D}}
}
\]
\end{theorem}
\begin{proof}
Recall $\mathbb{D}$, the left adjoint to the forgetful functors
$i^*$, see Lemma \ref{lem:leftadjoint}. That the left adjoints commute is
immediate from the definitions. It follows automatically that the
square of right adjoints also commutes.

The lower adjunction is a Quillen equivalence by \cite[Proposition 4.2.4]{BGKS}.
The weak equivalences and fibrations for
$K_a^{dual} \cell S_a^\bullet \leftmod$ and
$d \acal(\ccal)$ are defined
in terms of the functors $i^*$.
It is also routine to check that $i^*$ preserves
cofibrant objects.
Hence the adjunction
between $d \acal(\ccal)$ and
$ K_a^{dual} \cell  R_a^\bullet \leftmod$
is a Quillen equivalence by the
same argument as in Lemma \ref{lem:torusfixed}.
\end{proof}

We summarise this section with the following result.

\begin{corollary}\label{cor:toralconclusion}
There is a zig-zag of Quillen equivalences between
the model category $d \acal(\ccal)$ and the model category of toral $O(2)$--spectra,
$\ccal \osp$. Furthermore, these Quillen equivalences are compatible with the two forgetful functors
\[
i^* \co \ccal \osp \to \tsp \quad \textrm{and} \quad
i^* \co d \acal(\ccal) \to d \acal(\torus).
\]
Hence the algebraic forgetful functor correctly models the spectrum--level
forgetful functor.
\end{corollary}

\begin{rmk}\label{rmk:monoidal}
As with \cite{BGKS}, all of the Quillen equivalences of
this section are in fact symmetric monoidal, giving us a
classification of ring objects in $\ccal \osp$
in $d \acal(\ccal)$. However the classification of the
dihedral part of $\osp$ is not monoidal, so we do not
obtain a monoidal classification of $\osp$ overall.
\end{rmk}

\section{Dihedral spectra}\label{sec:dihedral}

In this section we find a model category based on chain complexes of $\qq$--modules
that is Quillen equivalent to the model category of dihedral spectra
$\dihedsp$. The material we present is an updated version of the preprint \cite{barnesdihedral}.
An early draft of this paper purported to give a symmetric monoidal
Quillen equivalence. That result relied on there being a
commutative ring $G$--spectrum weakly equivalent
to $e_\dcal S$. A correct interpretation
of such a result would require careful use of $N_\infty$--spectra
such as in \cite{BlumbergHillNorms} (and possible further work in that area).
Hence we leave monoidal considerations to future work.

\subsection{The dihedral model}\label{subsec:dihedralmodel}

The paper \cite{greo2} constructs an algebraic model for the
homotopy category of dihedral spectra.
We call this category $\acal(\dcal)$
and write $dg\acal(\dcal)$ for the category of differential graded objects
in $\acal(\dcal)$.
In this subsection we recap the definition of
that category and equip it with a model structure.
The category $\acal(\dcal)$ can also be described as the category of rational $O(2)$--Mackey functors
with support in the dihedral groups, see \cite[Examples C(iii)]{greratmack}.

Recall that $W$ is used to denote the group of order two.
For $R$ a ring, let $\ch(R)$ denote the category of chain complexes of $R$--modules.
We use $\qq[W]$ to denote the rational group ring of $W$.
We will often consider $\qq$--modules as
objects of $\qq[W] \leftmod$ with trivial $W$--action
without comment or decoration.

\begin{definition}
We define a category called $\acal(\dcal)$.
An object $M$
consists of the following data: a $\qq$--module $M_{\infty}$,
a collection of $\qq[W]$--modules $M_k$ for $k \geqslant 1$
and a map of $\qq [W]$--modules
$\sigma_M \co M_\infty \to \colim_n \prod_{k \geqslant n} M_k$.

A map $f \co M \to M'$ in this category consists of
a map $f_\infty \co M_{\infty} \to M_{\infty}'$
in $\qq \leftmod$ and maps
$f_k \co M_k \to M_k'$
in $\qq[W] \leftmod$
making the square below commute.
\[\xymatrix@C+0.5cm{
M_{\infty} \ar[r]^(0.4){\sigma_M} \ar[d]_{f_\infty} &
\colim_n \prod_{k \geqslant n} M_k
\ar[d]^{\colim_n \prod_{k \geqslant n} f_k} \\
M_{\infty}' \ar[r]^(0.4){\sigma_{M'}} &
\colim_n \prod_{k \geqslant n} M_k'
}\]
We will also write $\tails(M)$ for
$\colim_n \prod_{k \geqslant n} M_k$.
\end{definition}

\begin{definition}\label{def:dihedralmodel}
Let $dg \acal(\dcal)$ be the category of chain complexes
in $\acal(\dcal)$. We call this the \textbf{algebraic
model for dihedral spectra}.
We shall also need to use
$g \acal(\dcal)$,
the category of graded objects in
$\acal(\dcal)$.
\end{definition}

We see that an object $M$ of $dg \acal(\dcal)$
consists of a rational chain complex $M_{\infty}$
and a collection $M_k \in \ch (\qq [W])$ for $k \geqslant 1$
with a map of chain complexes of $\qq [W]$--modules
$\sigma_M \co M_\infty \to \colim_n \prod_{k \geqslant n} M_k$.
A map $f$ in this category consists of a map $f_\infty \in \ch(\qq)$
and maps $f_k \in \ch(\qq [W])$
such that the analogous square to the above definition commutes.

We want to show how to construct an object of $\acal(\dcal)$
from a rational $O(2)$--spectrum.
We first need to discuss some more idempotents of rationalised
Burnside rings.

Recall the idempotents $e_\ccal$, $e_\dcal$ and $e_n$ for $n \geqslant 1$
from Definition \ref{def:o2idem}.
We also have $f_n = e_\dscr - \Sigma_{k=1}^{n-1} e_k$.

\begin{lemma}\label{lem:subgroupidem}
Let $D^h_{2n}$ be a dihedral subgroup of $O(2)$ of order $2n$.
Then for each $k \mid n$ the rational Burnside ring of $D_{2n}^h$ has
idempotents $e_{C_k}$ and $e_{D_{2k}}$. The collection of idempotents
$e_{C_k}$ and $e_{D_{2k}}$ for $ k \mid n$ gives a maximal
orthogonal decomposition of the identity .

The inclusion map $D_{2n}^h \to O(2)$ induces the following map
$A(O(2)) \to A(D_{2n}^h)$
$$
\begin{array}{rcll}
e_\ccal  &  \mapsto  & \Sigma_{k \mid n} e_{C_k}     \\
e_\dcal  &  \mapsto  & \Sigma_{k \mid n} e_{D_{2k}}  \\
e_k      &  \mapsto  & e_{D_{2k}}   & k \mid n       \\
e_k      &  \mapsto  & 0            & k \nmid n
\end{array}
$$
\end{lemma}
\begin{proof}
Consider  the induced map
$\fcal D_{2n}^h / D_{2n}^h \to \fcal O(2)/O(2)$ (which sends the cyclic groups to $SO(2)$)
and use tom Dieck's isomorphism.
\end{proof}

The following definition and theorem
are taken from \cite{greo2}.
Note that for any compact Lie group $G$ and closed subgroup $H$,
the action of $N_G H/H$ on $G/H$
induces an action of $N_G H/H$ on
$[G/H_+,X]^G_* \cong \pi_*^H(X)$.

\begin{definition}\label{def:dihedralmackey}
Let $X$ be an $O(2)$--spectrum with rational homotopy groups.
We let $\underline{\pi}_*^\dcal(X)$ denote the following object of
$g\acal(\dcal)$.
Let $k \geqslant 1$ and define
\[
\underline{\pi}_*^\dcal(X)_k = e_{D_{2k}} \pi_*^{D_{2k}^h}(X) \quad \quad
\underline{\pi}_*^\dcal(X)_\infty = \colim_n ( f_n \pi_*^{O(2)}(X) ).
\]
Note that $\underline{\pi}_*^\dcal(X)_k$ is
isomorphic to the homotopy groups of the $D_{2k}$--geometric fixed points of $X$.
Whenever $k \geqslant n$, there is a map
\[f_n \pi_*^{O(2)}(X)  \longrightarrow
e_{D_{2k}} \pi_*^{D_{2k}^h}(X) \]
induced from the inclusion $D_{2k}^h \to O(2)$
and multiplication by $e_{D_{2k}}$.
Thus we obtain a map
\[
f_n \pi_*^{O(2)}(X)  \longrightarrow
\prod_{k \geqslant n } e_{D_{2k}} \pi_*^{D_{2k}^h}(X)
\]
Taking colimits over $n$ defines the structure map $\sigma$
of $\underline{\pi}_*^\dcal(X)$.

Since any fibrant object of $\dihedsp$ has rational homotopy
groups, this construction defines a functor
\[
\underline{\pi}_*^\dcal \co \ho \left( \dihedsp \right) \longrightarrow g\acal(\dcal).
\]
\end{definition}

Thus one has a map of graded $\qq$--modules
\[
[X,Y]^{\dcal O(2)}_* \to
\hom_{g\acal(\dcal)} ( \underline{\pi}_*^\dcal( X),
\underline{\pi}_*^\dcal(Y) ).
\]
This fits into an Adams short exact sequence as below, see
\cite[Corollary 5.5]{greo2}.

\begin{theorem}\label{thm:dihedraladams}
Let $X$ and $Y$ be $O(2)$--spectra with rational homotopy groups.
Then there is a short exact sequence as below.
\[
0 \to
\ext( \underline{\pi}_*^\dcal( \Sigma X),
\underline{\pi}_*^\dcal(Y) ) \to
[X,Y]^{\dcal O(2)}_* \to
\hom_{g\acal(\dcal)} ( \underline{\pi}_*^\dcal( X),
\underline{\pi}_*^\dcal(Y) ) \to 0
\]
\end{theorem}

\subsection{Adjunctions and model structures}\label{subsec:adjunct}

We now introduce a particularly useful construction, ${\bignplus}$.
We will soon see that this construction is an explicit description
of the `global sections' of an object of $dg\acal(\dcal)$.

\begin{definition}
Let $N \geqslant 1$ and take $M \in dg\acal (\dcal)$.
Then $\bignplus_N M$ is defined as the following pullback
in the category of $\ch(\qq [W])$--modules.
\[\xymatrix@C+0.5cm{
\bignplus_N M
\ar[r]^{\alpha_{M,N}}
\ar[d] &
\prod_{k \geqslant N} M_k
\ar[d] \\
M_\infty
\ar[r] & \tails(M)
}\]
Define $\bignplus_N^W M$ to be $(\bignplus_N M)^W$, the $W$--fixed points of
$\bignplus_N M$.
\end{definition}

It follows immediately from the definition above that
$\bignplus_N$ and $\bignplus_N^W$ are exact functors.
Furthermore,  there are natural isomorphisms
\vskip-0.4cm
{\setstretch{1.3}
\[
\begin{array}{rcl}
\bignplus_N M & \cong & \bignplus_{N+1} M \oplus M_N
\\
M_\infty & \cong  & \colim_N \bignplus_N M
\\
\tails(M)^W & \cong & \colim_n \prod_{k \geqslant n} (M_k^W).
\end{array}
\]}

The notation $\bignplus_N$ is to make the reader think
of some combination of a direct product and a direct sum.
Indeed if $M_\infty=0$, then $\bignplus_N M = \bigoplus_{k \geqslant N} M_k$.
Whereas if $M_\infty=\tails(M)$ and $\sigma=\id$,
then $\bignplus_N M = \prod_{k \geqslant N} M_k$.

Our first use of ${\bignplus}_N$ is to give a construction of limits
in $\acal(\dcal)$.

\begin{lemma}
The category $dg\acal (\dcal)$ contains all small limits and colimits.
\end{lemma}
\begin{proof}
Let $M^i$ be a small diagram of objects of $dg \acal (\dcal)$.
Define
\[
(\colim_i M^i)_\infty = \colim_i (M^i_\infty)
\quad \textrm{and} \quad
(\colim_i M^i)_k = \colim_i (M^i_k).
\]
The structure map for $\colim_i M^i$
is induced by the composite below.
\[
M^i_\infty \longrightarrow
\colim_n \prod_{k \geqslant n} M_k^i \longrightarrow
\colim_n \prod_{k \geqslant n} \colim_i M_k^i
\]
For limits, we define
\[
(\lim_i M^i)_k = \lim_i (M^i_k) \quad \textrm{and} \quad
(\lim_i M^i)_\infty = \colim_N \lim_i ({\bignplus}_N^W M^i).
\]
The structure map of $\lim_i M^i$
is the composite below, where the middle map is induced by
the maps
$\alpha_{M^i,N}^W \co {\bignplus}_N^W M^i \to \prod_{k \geqslant N} M_k^i$.
\[
(\lim_i M^i)_\infty =
\colim_N \lim_i \left ( {\bignplus}_N^W M^i \right) \longrightarrow
\colim_N \lim_i \prod_{k \geqslant N} M^i_k =
\tails (\lim_i M^i)
\]
It is routine to check that these constructions give the
colimit and limit.
\end{proof}

We will need the following fact to construct the model structure on
$\acal(\dcal)$.

\begin{lemma}\label{lem:globseccolim}
The functors
${\bignplus}_N$ and ${\bignplus}_N^W$
preserve filtered colimits
for all $N \geqslant 1$.
\end{lemma}
\begin{proof}
One checks injectivity and surjectivity
of the canonical map
\[
\colim_i {\bignplus}_N M^i
\longrightarrow
{\bignplus}_N \colim_i M^i
\]
by first dealing with the term at infinity, then dealing with the finite
number of terms that are not determined
by the term at infinity.
\end{proof}

We introduce a collection of useful adjunctions
relating $dg\acal(\dcal)$ to rational chain complexes
and $W$--equivariant rational chain complexes.

\begin{definition}\label{def:manyadjoints}
Let $A$ be a rational chain complex, $R \in \ch(\qq[W])$
and $M \in dg\acal(\dcal)$.

Define $i_k R$ to be the object of $dg\acal(\dcal)$
with $(i_k R)_\infty =0$, $(i_k R)_n = 0$ for $n \neq k$ and
$(i_k R)_k = R$. Now define $p_k$ by setting $p_k M = M_k \in \ch(\qq[W])$.
Then $i_k$ is both right and left adjoint to $p_k$.
\[
i_k : \ch(\qq[W])  \adjunct
dg\acal(\dcal) : p_k \qquad \quad
p_k : dg\acal(\dcal) \adjunct
\ch(\qq[W])  : i_k
\]
Let $p_\infty M = M_\infty \in \ch(\qq)$
and define $(i_\infty A)_\infty = A$ and $(i_\infty A)_k = 0$. Then we have an adjunction
\[p_\infty : dg\acal(\dcal) \adjunct
\ch(\qq) : i_\infty .\]

We set $cA$ to be the object of $dg\acal(\dcal)$
with $cA_k = A = cA_\infty$ and structure map induced by the
diagonal map $A \to \prod_{k \geqslant 1} A$. Then we have
the `constant sheaf' and `global sections'
adjunction
\[c : \ch(\qq) \adjunct
dg\acal(\dcal) : {\bignplus}_1^W \]
\end{definition}

We put a model structure on $dg\acal(\dcal)$.
We use the functors $i_k$ and $c$ above to create the
generating sets.
Let $I_\qq$ and $J_\qq$ denote the sets of generating
cofibrations and
acyclic cofibrations for the projective model structure on
rational chain complexes, see \cite[Section 2.3]{hov99}.
Similarly we have generating sets $I_{\qq [W]}$ and $J_{\qq [W]}$ for
$\ch(\qq[W])$.

\begin{proposition}
Define a map $f$ in $dg\acal(\dcal)$ to be
a weak equivalence or fibration
if $f_\infty$ and each $f_k$ is
a homology isomorphism or surjection.
These classes define a cofibrantly generated and proper
model structure on the category $dg\acal (\dcal)$.

The generating cofibrations $I$ are the collections
$cI_\qq$ and $i_k I_{\qq [W]}$ for $k \geqslant 1$. The generating
acyclic cofibrations $J$ are
$cJ_\qq$ and $i_k J_{\qq [W]}$ for $k \geqslant 1$.
\end{proposition}
\begin{proof}
Lemma \ref{lem:globseccolim} shows that ${\bignplus}^W_1$
preserves filtered colimits. The required smallness conditions on
the generating sets follows immediately.

The rest of the proof is routine.
As an example of the kind of argument we need to make,
we prove that the acyclic fibrations
are the maps with the right lifting property with
respect to $I$. Let $f \co A \to B$ be such a map.
Using the adjunctions of
Definition \ref{def:manyadjoints} it follows that each $f_k \co A_k \to B_k$
is a surjection and a homology isomorphism,
as is $\bignplus_1^W f \co \bignplus_1^W A \to \bignplus_1^W B$.

Since $\bignplus_N^W A \cong \bignplus_{N+1}^W A \oplus A_N^W$
it follows that each $\bignplus_N^W f$ is a surjection and a homology isomorphism
for each $N \geqslant 1$.
Taking colimits over $N$ we see that $f_\infty$ is a
surjection and homology isomorphism.

Left properness is immediate because colimits are defined term--wise.
For right properness the only difficulty occurs at infinity,
but the same method as in the preceding paragraph suffices,
using exactness of $\bignplus_N^W$ to see that it preserves surjections
and homology isomorphisms.
\end{proof}

Note that with this model structure, the adjunctions
$(i_k, p_k)$ and $(c, \bignplus_N^W)$
are Quillen pairs between rational chain complexes
(with a $W$--action in the first case) and
$\acal(\dcal)$.

\begin{lemma}\label{lem:alggens}
The collection $i_k \qq[W]$ for $k \geqslant 1$
and $c \qq$ are a set of homotopically compact,
cofibrant and fibrant generators for this category.
\end{lemma}
\begin{proof}
Every object of $dg\acal (\dcal)$ is fibrant and these objects are
images of cofibrant objects under left Quillen functors.
Homotopy compactness is simple to check for $i_k \qq[W]$.
For $c \qq$ it relies on the fact that $\bignplus_1^W$
commutes with arbitrary coproducts (as they are filtered
colimits of finite products).

Assume that $[\sigma, M]^{\acal(\dcal)}_*=0$
for each $\sigma$ in the collection.
It follows that each $M_k$ must be acyclic as
\[
0= [i_k \qq[W], M]^{\acal(\dcal)}_* \cong
[\qq[W], M_k]^{\ch(\qq[W])}_*
\]
It follows that the canonical map $M \to i_\infty M_\infty$,
which is the identity at infinity and zero elsewhere,
is a weak equivalence.
Thus we know that the graded groups
\[
[cQ, M]^{\acal(\dcal)}_*
\cong
[cQ, i_\infty M_\infty]^{\acal(\dcal)}_*
\cong
[\qq, M_\infty]^{\qq}_* \cong \h_*(M_\infty)
\]
are zero. Hence $M_\infty$ is acyclic
and $M \to 0$ is a weak equivalence.
\end{proof}

It is time we make our analogy to sheaves clear.
In particular this explains why the construction of limits
in $\acal(\dcal)$ are more complicated than colimits:
$\acal(\dcal)$ is essentially a category of sheaves
described in terms of stalks.

\begin{definition}
Let $\pscr$ be the space $\fcal O(2)/O(2) \setminus \{ SO(2) \}$
($\pscr$ for points). Let $\oscr$ be the constant sheaf of $\qq$ on $\pscr$
considered as a sheaf of rings.
Let $W \oscr \leftmod$ denote the category of
$W$--equivariant objects and $W$--equivariant maps in
$\oscr \leftmod$.
\end{definition}

To specify an $\oscr$--module $M$ one only needs to give
the stalks at the points $k$ and $\infty$ and a
map of $\qq$--modules $M_\infty \to \tails(M)$.
The global sections of $M$ are then given by
$\bignplus_1 M$. Hence any object of $dg\acal(\dcal)$
defines an object of $W \oscr \leftmod$.
We call this functor $\tinc$ and see that it is full and faithful.
Thus we can view $dg\acal(\dcal)$ as a full subcategory of
$W \oscr \leftmod$.

The inclusion functor
has a right adjoint called `$\tfix$'.
On a $W$--equivariant $\oscr$--module $M$, $\tfix(M)_k = M_k$,
$\tfix(M)_\infty = M_\infty^W$ and the structure map
is $M_\infty^W \to M_\infty \to \tails(M)$.
\[
\tinc : dg\acal(\dcal)
\adjunct
dg W \oscr \leftmod : \tfix
\]
In particular, one could also describe the limit of some diagram $M^i$
in $dg\acal (\dcal)$ as $\tfix \lim_i \text{inc } M^i$,
where the limit on the right is taken in the category of $W \oscr$--modules.

\begin{rmk}
The adjunction
$\tinc : dg\acal(\dcal)
\adjunct
W \oscr \leftmod : \tfix$
can be used to put a model structure
on $W \oscr \leftmod$. Define a map
$f$ to be a weak equivalence or fibration if
$\tfix f$ is. Then we have a new cofibrantly
generated model structure on  $W \oscr \leftmod$.
With this model structure
the adjunction $(\tinc,\tfix)$ becomes a Quillen equivalence.
\end{rmk}

\subsection{The dihedral comparison}

In this subsection we give the proof that $dg\acal(\dcal)$ and
$\dihedsp$ are Quillen equivalent.
Since we are not considering monoidal products,
we use the tilting theorem of Schwede and Shipley, \cite[Theorem 5.1.1]{ss03stabmodcat}.
Recall that a set of \textbf{tiltors} is a set of homotopically compact generators for the
homotopy category such that 
$[T,T']_*$ is concentrated in degree zero
for any $T$, $T'$ in the set.
One could use a similar method to that of \cite{barnesfinite}, but the
argument given below is somewhat simpler.

\begin{lemma}\label{lem:dihedraltiltors}
The model category $\dihedsp$ has a set of tiltors given by
the following countably infinite collection of cofibrant-fibrant objects.
Let $\fibrep_\dcal$ denote fibrant replacement in $\dihedsp$ and define
\[
\gcal_{top} = \{ \fibrep_\dcal S^0 \} \cup \{ \fibrep_\dcal e_H O(2)/H_+  \ | \ H \in \dcal \setminus \{O(2) \} \}.
\]
\end{lemma}
\begin{proof}
For any two dihedral subgroups $H$ and $K$,
the set of maps
\[ [O(2)/H_+ ,O(2)/K_+ ]^{\dcal O(2)}_*\]
is concentrated in degree zero,
where $[- ,-]_*^{\dcal O(2)}$ denotes maps in the homotopy
category of $\dihedsp$.
Hence it follows that, in the homotopy category, maps between elements of $\gcal_{top}$
are concentrated in degree 0.

To show that $\gcal_{top}$ generates, we must prove that if $X$ is an object of
$\dihedsp $ such that
$[\sigma ,X]_*^{\dcal O(2)}=0$
for all $\sigma \in \gcal_{top}$,
then $X \to *$ is a $\pi_*$--isomorphism.
Let $X \in \dihedsp $ be fibrant, by Theorem \cite[IV.6.13]{mm02}
$\pi_*^{H}(X)=0$ for any $H \in \ccal$
and we see immediately that
\[
\pi_*^{O(2)}(X) = [S^0 ,X]_*^{O(2)} = [S^0 ,X]_*^{\dcal O(2)}=0
\]
Let $H$ be a finite dihedral group.
By \cite[Examples C(i)]{greratmack}
there is a natural isomorphism
\[
\pi_*^H(X) \cong \bigoplus_{(K) \leqslant H} (e_K \pi_*^K(X))^{W_H K}
\]
Since we have assumed that
$e_K \pi_*^{K}(X)=[e_K O(2)/K_+ ,X]_*^{\dcal O(2)}$ is zero
for each finite dihedral $K$,
$\pi_*^H(X)=0$.
Hence our set generates the homotopy category.
Homotopy compactness follows from the isomorphisms
\[
[e_K O(2)/K_+, X]^{\dcal O(2)}_* = [e_K O(2)/K_+, X]^{O(2)}_* = e_K \pi_*^K(X)
\]
which hold whenever $X$ is fibrant in $\dihedsp$.
\end{proof}

We identify a ringoid $\underline{R}$ from our algebraic model and give an algebraic version of the
tilting result we want for dihedral spectra.

\begin{definition}
Define a set of objects $\gcal_a$ of $dg\acal(\dcal)$
\[
\gcal_a=\{ c \qq \} \cup \{ i_k \qq[W] \mid k \geqslant 1 \}.
\]
Let $\underline{R}$ denote the ringoid given by taking the full subcategory of $\acal(\dcal)$
on the object set $\gcal_a$, considered as a category
enriched over rational vector spaces.
A \textbf{module over $\underline{R}$} is a contravariant additive functor from
$\underline{R}$ to rational vector spaces.
\end{definition}

By Lemma \ref{lem:alggens}, $\gcal_a$ is a set of homotopically
compact cofibrant-fibrant generators for the homotopy category of $dg\acal(\dcal)$.
A standard variation of the tilting theorem (using rational chain complexes instead
of symmetric spectra) gives the following result.
\begin{proposition}\label{prop:algmorita}
The model category of chain complexes of modules over $\underline{R}$
(with fibrations the objectwise surjections and weak equivalences the objectwise
homology isomorphisms)
is Quillen equivalent to $dg\acal(\dcal)$.
\end{proposition}

We now prove that $\underline{R}$ is isomorphic to the endormorphism ringoid of
$\gcal_{top}$.

\begin{lemma}\label{lem:dihedralcalc}
The functor $\underline{\pi}_*^\dcal$ induces an isomorphism
of categories (enriched over rational vector spaces)
from the full subcategory of $\ho (\dihedsp)$
with object set $\gcal_{top}$ to $\underline{R}$.
\end{lemma}
\begin{proof}
This is a series of routine calculations using \cite{greo2}.
Let $H$ and $K$ be finite dihedral groups with $|H|=2k$ and $|K|=2m$
and let $\sigma_H = \fibrep_\dcal e_H O(2)/H_+$.
Then
\[
\underline{\pi}_*^\dcal(S_\dcal)  =  c \qq \quad \text{and} \quad
\underline{\pi}_*^\dcal(\sigma_H) = \underline{\pi}_*^\dcal(e_H O(2)/H_+)  =  i_k \qq[W]
\]
The functor $\underline{\pi}_*^\dcal$
from Definition \ref{def:dihedralmackey} gives maps as below.
These maps are isomorphisms by Theorem \ref{thm:dihedraladams}
as $i_k \qq[W]$ and $c \qq$ are projective, see the proof of \cite[Remark 4.3]{greo2}.
The equalities on the right hold as the objects of $\gcal_a$ are cofibrant, fibrant and concentrated
in degree zero.
\[
\begin{array}{rclll}
{[S_\dcal, S_\dcal]^{\dcal O(2)}_* }& \overset{\cong}{\longrightarrow} &
{[c \qq , c \qq]^{dg\acal(\dcal)}_* }
= \acal(\dcal) (c \qq , c \qq) \\
{[\sigma_H , S_\dcal]^{\dcal O(2)}_* }& \overset{\cong}{\longrightarrow} &
{[i_k \qq[W] , c \qq]^{dg\acal(\dcal)}_* }
= \acal(\dcal) (i_k \qq[W], c\qq)\\
{[S_\dcal, \sigma_H ]^{\dcal O(2)}_* }& \overset{\cong}{\longrightarrow} &
{[c \qq, i_k \qq[W]  ]^{dg\acal(\dcal)}_* }
= \acal(\dcal) (c\qq , i_k \qq[W])\\
{[\sigma_H , \sigma_H ]^{\dcal O(2)}_* }& \overset{\cong}{\longrightarrow}&
{[i_k \qq[W] , i_k \qq[W]  ]^{dg\acal(\dcal)}_* }
= \acal(\dcal) (i_k \qq[W] , i_k\qq[W])\\
{[\sigma_K , \sigma_H ]^{\dcal O(2)}_* }& \overset{\cong}{\longrightarrow} &
{[i_m \qq[W] , i_k \qq[W]  ]^{dg\acal(\dcal)}_* }
= \acal(\dcal) (i_m \qq[W] , i_k\qq[W]). \\
\end{array}
\]
\end{proof}

We can now give the classification theorem for dihedral spectra.

\begin{theorem}\label{thm:dihedralconclusion}
The model categories $\dihedsp$
and $dg\acal(\dcal)$ are Quillen equivalent.
Hence the homotopy categories of $\dihedsp$
and $dg\acal(\dcal )$ are equivalent.
\end{theorem}
\begin{proof}
The model category $\dihedsp$ is simplicial, cofibrantly generated, proper
and stable as these properties are preserved by the localisations we have applied to $O(2) \Sp$.

Lemma \ref{lem:dihedralcalc} gives an isomorphism of ringoids,
hence \cite[Theorem 5.1.1]{ss03stabmodcat} implies that
$\dihedsp$ is Quillen equivalent to the model category of chain complexes
of $\underline{R}$--modules.
The result then follows by Proposition \ref{prop:algmorita}.
\end{proof}

\paragraph{Acknowledgements}
The author would like to thank Constanze Roitzheim
for more help than it is possible to describe in a sentence,
John Greenlees for numerous productive conversations and helpful suggestions
and Magdalena K\c{e}dziorek for many helpful comments on a draft of this work.
The author also wishes to thank the anonymous referee for many helpful statements
that have shortened proofs and improved the paper.

The author most gratefully acknowledges funding from
the Engineering and Physical Sciences Research Council,
grant numbers  EP/H026681/1 and EP/M009114/1.

\addcontentsline{toc}{part}{Bibliography}
\bibliography{extendbib}
\bibliographystyle{alpha}

\end{document}